\documentclass[12pt, a4paper]{article}
\usepackage{amsfonts}
\usepackage{mathtext}
\usepackage{amssymb,latexsym,amsmath,amscd,amsthm,hyperref}
\usepackage[english]{babel}
\usepackage[left=3cm,right=2.5cm,top=2.5cm,bottom=2.5cm]{geometry} 
\usepackage{graphicx}
\usepackage{enumitem}
\usepackage{graphicx}
\usepackage{tikz}
\usepackage{fancyhdr}
\newtheorem{theorem}{Theorem}[section]
\newtheorem{lemma}[theorem]{Lemma}
\newtheorem{corollary}[theorem]{Corollary}

\newtheorem{algorithm}[theorem]{Algorithm}

\newtheorem{remark}{Remark}

\newcommand{\smallo}{{\scriptstyle \mathcal{O}}}

\begin{document}

\title{Optimal approximation order of piecewise constants on convex partitions}

\author{
	Oleg Davydov
		\thanks{Department of Mathematics, Justus Liebig University, Arndtstrasse 2, 35392 Giessen, Germany}
	\and
	Oleksandr Kozynenko
		\thanks{Department of Mechanics and Mathematics, Oles Honchar Dnipro National University, pr. Gagarina 72, 49010 Dnipro, Ukraine}
	\and
    Dmytro Skorokhodov
    	\thanks{Department of Mechanics and Mathematics, Oles Honchar Dnipro National University, pr. Gagarina 72, 49010 Dnipro, Ukraine}
}

\maketitle

\begin{abstract}
We prove that the error of the best nonlinear $L_p$-approximation by piecewise constants on convex partitions is $\mathcal{O}\big(N^{-\frac{2}{d+1}}\big)$,
where $N$ is the number of cells, for all functions in the Sobolev space $W^2_q(\Omega)$ on a cube $\Omega\subset\mathbb{R}^d$, $d\geqslant 2$, as soon as 
$\frac{2}{d+1} + \frac{1}{p} - \frac{1}{q}\geqslant 0$. The approximation order $\mathcal{O}\big(N^{-\frac{2}{d+1}}\big)$ is achieved on a polyhedral partition 
obtained by anisotropic refinement of an adaptive dyadic partition. Further estimates of the approximation order  from the above and
below are given for various Sobolev and Sobolev-Slobodeckij spaces  $W^r_q(\Omega)$ embedded in $L_p(\Omega)$, some of which also improve the
standard estimate $\mathcal{O}\big(N^{-\frac 1d}\big)$ known to be optimal on isotropic partitions.
\end{abstract}

\section{Introduction}

Nonlinear approximation with piecewise polynomials seeks a better approximation in comparison to linear methods by adapting the 
spline  to the local behavior of the data assuming it stems from a function in a smoothness space. Initial multivariate results of this type 
(Birman and Solomyak~\cite{BS67}) were obtained 
by adjusting partitions to the local Sobolev energy of the function, but more recent research focuses on the $n$-term approximation with appropriate wavelet-like bases 
\cite{DeVore98,Temlyakov11} and adaptive algorithms that recursively reduce the error function \cite{DeVore09}. 
Very little beyond the results of \cite{BS67} is known in the  multivariate ``free partition'' setting, see e.g.\ \cite[Section 6.5]{DeVore98}.

In this paper we continue investigations in~\cite{D11,D12} and study approximation properties of piecewise constants on arbitrary convex partitions
that are freely adjusted to functions in Sobolev spaces. This adjustment, and notably the use of anisotropic convex polyhedral partitions allows to nearly double the 
approximation order in comparison to ``isotropic'' partitions employed in \cite{BS67}. Interestingly, these results do not extend to 
higher order splines, where introducing ``anisotropy'' to a partition does not bring any improvement in the order of approximation \cite{D12}
(see also~\cite{BabBabParSko}), and a promising approach is to consider piecewise polynomials 
on several overlaid polyhedral partitions \cite{D14}.

In comparison to~\cite{D12} we obtain improved convergence orders on much wider classes of Sobolev spaces in the setting 
typical for nonlinear approximation where the $L_p$-metric error is considered for $L_q$-metric classes of functions with $q<p$.  
To achieve this we generalize the cell counting techniques introduced in \cite[Theorem~2.1]{BS67}. 

Let $\Omega\subset\mathbb{R}^d$, $d\geqslant 2$, be a bounded domain. A finite collection $\triangle$ of subdomains $\omega\subset \Omega$, referred to as {\it cells} throughout the paper, is called a \textit{partition} of $\Omega$ provided that $\omega \cap \omega' = \varnothing$, for every $\omega,\omega'\in\triangle$, $\omega \ne \omega'$, and $\sum_{\omega \in \triangle} \left|\omega\right| = \left|\Omega\right|$, where $\left|\cdot\right|$ stands for the Lebesgue measure in $\mathbb{R}^d$. 
We call a partition $\triangle$ of $\Omega$ \textit{convex} if every cell $\omega\in\triangle$ is convex. For $N\in\mathbb{N}$, denote by $\mathfrak{D}_N$ the set of all convex partitions of $\Omega$ comprising at most $N$ cells. 
For simplicity, we assume that $\Omega$ is a cube in $\mathbb{R}^d$, even if the results of the paper may be naturally extended to more general domains by for example splitting them
into a finite number of affine images of cubes.

For $1\leqslant q \leqslant \infty$ and $r\in\mathbb{N}$, let $W^r_q \left( \Omega\right)$ be the Sobolev space of measurable functions $f:\Omega\to\mathbb{R}$ endowed with the standard norm and seminorm
\[
\|f\|_{W^r_q(\Omega)} := \sum\limits_{{\bf k}\in\mathbb{Z}_+^d\,:\,|{\bf k}| \leqslant r} \left\|D^{\bf k} f\right\|_{L_q\left(\Omega\right)}\quad\textrm{and}\quad |f|_{W^r_q(\Omega)} := \sum\limits_{{\bf k}\in\mathbb{Z}_+^d\,:\,|{\bf k}| = r} \left\|D^{\bf k} f\right\|_{L_q\left(\Omega\right)},
\]
where ${\bf k} = \left(k_1,\ldots,k_d\right)\in\mathbb{Z}_+^d$ is non-negative integer multi-index, $|{\bf k}|:=k_1+\cdots+k_d$, and $D^{\bf k}f$ denotes the 
(generalized) partial
derivative of order ${\bf k}$. We set $W^0_q(\Omega):=L_q(\Omega)$. For $f\in W^1_1(\Omega)$, we denote by $\nabla f$ the gradient of $f$. 

For $1 \leqslant q < \infty$ and $r \in \mathbb{R}_+\setminus\mathbb{Z}_+$, we consider the Sobolev-Slobodeckij space \cite{Slobodeckiy58}:
\[
\displaystyle W^r_q (\Omega) := \left\{f\in W^{\lfloor r\rfloor }_q(\Omega) : |f|_{W^r_q(\Omega)}< \infty\right\},\quad
\|f\|_{W^r_q(\Omega)} := \|f\|_{W^{\lfloor r\rfloor}_q(\Omega)}+|f|_{W^r_q(\Omega)},
\]
where 
$$
|f|_{W^r_q(\Omega)} := \sup\limits_{{\bf k}\in\mathbb{Z}_+^d:|{\bf k}|=\lfloor r\rfloor } 
\Big( \int_\Omega \int_\Omega \frac{\left|D^{\bf k} f({\bf x}) - D^{\bf k} f({\bf y})\right|^q}{\left| {\bf x}-{\bf y}\right|^{\left\{r\right\}q+d}}\,
{\rm d}{\bf x}{\rm d}{\bf y}\Big)^{\frac{1}{q}},  $$
$\lfloor r\rfloor $ and $\{r\}$, respectively, are the integer and fractional parts of $r$, and $|{\bf x}|$ stands for the Euclidean norm of ${\bf x}\in\mathbb{R}^d$. 
Spaces $W^r_q(\Omega)$ coincide with the Besov spaces $B^r_{q,q}(\Omega)$ (see~\cite[p. 323]{Triebel_78}).

For a partition $\triangle$ of $\Omega$, we denote by $\mathcal{S}_0(\triangle)$ the space of piecewise constant functions $s:\Omega\to\mathbb{R}$ that are constant on
every cell $\omega\in\triangle$. For $1\leqslant p \leqslant\infty$, we define the error of the best $L_p$-approximation of a function 
$f\in L_p(\Omega)$ by piecewise
constant functions on partitions from $\mathfrak{D}_N$:
\[
E_N(f)_{p} := \inf\limits_{\triangle\in\mathfrak{D}_N}\inf\limits_{s\in \mathcal{S}_0\left(\triangle\right)}\left\|f-s\right\|_{L_p(\Omega)}.
\]
Since the partition $\triangle$ is allowed to depend on $f$, the optimal or near-optimal approximations $s$ do not belong to a linear space of functions when $N$ is fixed, 
and the methods of constructing them are referred to as nonlinear approximation \cite{DeVore98}.  There are very few results in the literature in this ``free partition'' setting, see \cite[Section 6.5]{DeVore98} and \cite{D11}. In what follows we will consider estimates for the order of approximation of functions $f\in W^r_q(\Omega)$ with $q < \infty$ only. Although results obtained in this paper hold true for functions $f\in W^r_\infty(\Omega)$, they also follow immediately from the results obtained in~\cite{D12}. This is the main motivation for us to avoid consideration of the case $q=\infty$.

Birman and Solomyak's result on piecewise polynomial approximation \cite[Theorem~3.2]{BS67} implies that 
\begin{equation}
	\label{BS}
	E_N(f)_p=\mathcal{O}\big(N^{-\frac{r}{d}}\big),\quad f\in W^r_q(\Omega),\quad 0 < r\leqslant 1,
\end{equation}
as soon as $q$ satisfies
\begin{equation}
	\label{BSr}
	\frac{r}{d} + \frac{1}p-\frac 1q > 0\quad\text{and}\quad q < \infty,
\end{equation}
where in the case $p=\infty$ we set $\frac{1}{p} := 0$. Note that the construction of a sequence of partitions $\left\{\triangle_N\right\}_{N=1}^\infty$ that attains the order in \eqref{BS} for $q<p$ is based on adaptively refined dyadic
subdivisions of $\Omega$. 
Considering $r>1$, it follows from the embedding $W^r_q(\Omega)\subset W^1_{q'}(\Omega)$, with $\frac 1{q'} =  \frac{r-1}{d}+\frac 1q$ 
(see~\cite[p. 328]{Triebel_78}) that
\begin{equation}
	\label{BSc}
	E_N(f)_p=\mathcal{O}\big(N^{-\frac 1d}\big),\quad f\in W^r_q(\Omega),\quad  r\geqslant 1,
\end{equation}
under the same restrictions \eqref{BSr}.

No improvement of the order $N^{-\frac 1d}$ is possible for any smooth non-constant function $f$ on ``isotropic'' partitions
(see~\cite{D11}). 
Recall that a sequence of convex partitions $\left\{\triangle_N\right\}_{N=1}^\infty$ is said to be  \emph{isotropic} if there
exists $\gamma > 0$ such that $\text{diam}\,(\omega) \leqslant \gamma \rho(\omega)$ for all $\omega\in \bigcup\limits_{N\in
	\mathbb{N}}\triangle_N$, where $\text{diam}\,(\omega)$ is the diameter of the cell $\omega$ and $\rho(\omega)$ is the diameter
of the largest $d$-dimensional ball contained in $\omega$. If this condition is not satisfied, then we say that the
partitions in the sequence are ``anisotropic''.
It was shown in~\cite{D11,D12} (and earlier in~\cite{Kochurov95} for $d=2$) that on a wider set of all convex partitions $\mathfrak{D}_N$
significantly better order of approximation can be achieved. More precisely, for $1\leqslant p \leqslant\infty$, 
\[
E_N(f)_p =\mathcal{O}\big(N^{-\frac 2{d+1}}\big),\quad f\in W^2_p(\Omega),
\]
which almost doubles the approximation order in comparison to $N^{-\frac 1d}$ when $d$ is large. This improvement in order is
obtained on a sequence of anisotropic partitions
constructed in two steps by first subdividing $\Omega$ uniformly into smaller sub-cubes and then splitting each of the
sub-cubes with the help of equidistant
hyperplanes orthogonal to the average gradient of $f$ on the sub-cube. Moreover, it is shown in \cite{D12} that $N^{-\frac 2{d+1}}$ is the 
saturation order of piecewise constant approximation on convex partitions in the following sense: 
if $E_N(f)_\infty=\smallo(N^{-\frac 2{d+1}})$ as $N\to\infty$
for a twice continuously differentiable function $f:\Omega\to\mathbb{R}$, then the Hessian of $f$ cannot be either positive definite or negative definite at any point in $\Omega$.
This result has been extended to $E_N(f)_p$ for all $1\leqslant p <\infty$ in~\cite{Kozynenko18}. 

In this paper we are interested in determining Sobolev and Sobolev-Slobodeckij spaces  $W^r_q(\Omega)$ for which the order of $E_N(f)_p$ is $\mathcal{O}\big(N^{-\frac 2{d+1}}\big)$ or 
at least $\smallo(N^{-\frac 1d})$, $N\to\infty$, such that anisotropic partitions provide superior convergence.

Our main result is the estimate (Theorem~\ref{theorem1})
\begin{equation}\label{main}
	E_N(f)_p\leqslant CN^{-\frac 2{d+1}}\|f\|_{W^2_q(\Omega)},\quad f\in W^2_q(\Omega),
\end{equation}
for all $q<\infty$ satisfying
\[
\frac{2}{d+1} + \frac{1}{p} - \frac{1}{q} \geqslant 0,
\]
where $C$ depends only on $d,p,q$.
The optimal order $N^{-\frac 2{d+1}}$ is achieved on a sequence of adaptively refined dyadic subdivisions that subsequently undergo anisotropic refinement 
(see Algorithm~\ref{algorithm1}). 
The proof of \eqref{main} relies on a generalization of Birman-Solomyak's cell counting techniques  \cite[Theorem~2.1]{BS67} to the case when the dyadic cells
may carry varying numbers of degrees of freedom (see Section 2), which seems to be of independent interest. 

Since the spaces $W^2_q(\Omega)$ are embedded into $L_p(\Omega)$ whenever $\frac{2}{d} + \frac{1}{p} - \frac{1}{q} \geqslant 0$ when $p<\infty$ 
or $\frac{2}{d} -\frac{1}{q} > 0$ when $p=\infty$ \cite[Theorem~4.12]{AdaFou}, this leaves open the question on the order of $E_N(f)_p$ as $N\to\infty$
for functions $f\in W^2_q(\Omega)$ when the parameters $1\leqslant p\leqslant \infty$ and $1\leqslant q<\infty$ satisfy
\begin{equation}
	\label{parameters_2}
	\left\{
	\begin{array}{ll}
		\frac{2}{d+1} + \frac{1}{p} < \frac 1q \leqslant \frac{2}{d} + \frac{1}{p}, & \textrm{if }\;p < \infty,\\ [6pt]
		\frac{2}{d+1} + \frac{1}{p} < \frac 1q < \frac{2}{d} + \frac{1}{p}, & \textrm{if }\;p = \infty.
	\end{array}\right.
\end{equation}
We partially fill this gap in Theorem~\ref{theorem2} by showing that under these conditions 
\[
E_N(f)_p=\mathcal{O}\left(N^{-d\left(\frac{2}{d} + \frac{1}{p} - \frac{1}{q}\right)}\right),\quad f\in W_q^2(\Omega).
\]
This proves in particular that
\[
E_N(f)_p=\smallo\left(N^{-\frac 1d}\right),\quad N\to\infty,\quad f\in W_q^2(\Omega),
\]
as soon as 
\[
\frac{2}{d} + \frac{1}{p} -\frac1q >\frac{1}{d^2}.
\]
For the remaining range of parameters $p$ and $q$ the estimate $E_N(f)_p = \mathcal{O}\left(N^{-\frac 1d}\right)$  that follows directly from~\eqref{BSc} remains the best known
for $f\in W_q^2(\Omega)$. 
The above results extend to Sobolev and Sobolev-Slobodeckij spaces $W^r_q(\Omega)$ with $r>2$ by embedding, see Corollary~\ref{cor}.

Finally, we consider the question whether the estimates \eqref{BS} and \eqref{BSc} can be improved for $r<2$. 
Theorem~\ref{lb} shows that for all $1 \leqslant q <\infty$ and $\frac{d}{q}<r<2$,
\[
\sup_{f\in W^r_q(\Omega)\atop \|f\|_{W^r_q(\Omega)}\leqslant 1} E_N(f)_\infty\ne\smallo\left(N^{-\frac{r}{d}}\right),\quad N\to\infty.
\]
(Recall that $r> \frac{d}{q}$ ensures that $W^r_q(\Omega)$ is embedded into $L_\infty(\Omega)$.) This shows in particular that \eqref{BS} cannot be improved for 
$p=\infty$. Another consequence is 
$$
\sup_{f\in W^r_q(\Omega)\atop \|f\|_{W^r_q(\Omega)}\leqslant 1} E_N(f)_\infty\ne\mathcal{O}\big(N^{-\frac 2{d+1}}\big)\quad
\text{if}\quad r < \frac{2d}{d+1}.$$
Hence an estimate of the type $E_N(f)_\infty\leqslant CN^{-\frac 2{d+1}}\|f\|_{W^r_q(\Omega)}$ cannot hold for all $f\in W^r_q(\Omega)$ when $r < \frac{2d}{d+1}$,
leaving only a small gap of $\frac{2d}{d+1}\leqslant r<2$ open for the possibility of extending \eqref{main} to the spaces $W^r_q(\Omega)$ with $r<2$ in the case $p=\infty$.

The paper is organized as follows. We devote Section~\ref{aux} to auxiliary results that generalize Theorem~2.1 in Birman-Solomyak's paper~\cite{BS67}.
The main results about the order of the best nonlinear approximation by piecewise constants are presented in Section~\ref{optorder} (estimates from above) and
Section~\ref{below} (estimates from below).

\section{Cell counting}
\label{aux}

Let $\Omega$ be a cube $(a_1,a_1+h)\times\cdots\times(a_d,a_d+h)$ in $\mathbb{R}^d$, $d\ge1$, with side length $h$.
We follow~\cite{BS67} to define dyadic subdivisions of $\Omega$. Let $\square$ be a partition of $\Omega$ into a finite number of open cubes. A partition
$\square^\prime$ of $\Omega$ is an {\it elementary extension} of $\square$ if it can be obtained from $\square$ by uniformly splitting some of its cubes into 
$2^d$ equal open cubes with halved side length. We
call a partition $\square$ of $\Omega$ a {\it dyadic subdivision} if it is obtained from the singleton partition $\{\Omega\}$ with the help of a finite number of elementary extensions. 

A crucial tool in establishing the main results of \cite{BS67} including the estimate \eqref{BS} is provided by Theorem~2.1 in~\cite{BS67} that gives an upper
bound on the number of cells in a specially constructed dyadic subdivision of $\Omega$. 

Similarly, the main ingredient for obtaining our estimate \eqref{main} 
is an upper bound on the number of cells in a specially constructed convex partition $\triangle$ of $\Omega$. The partition $\triangle$ is built in two steps by first
constructing a dyadic subdivision $\square$ of $\Omega$ and then splitting every cube $\omega\in \square$ anisotropically with the help of parallel hyperplanes
orthogonal to the average gradient of $f$ on $\omega$, where the number of hyperplanes depends on the size of $\omega$ relative to $\Omega$, see Algorithm~\ref{algorithm1}.
In order to obtain this bound we extend the techniques of \cite{BS67} to the case when each dyadic subcube in $\square$ is assigned a different number of degrees of freedom 
depending on its size.

As in \cite{BS67} we say that a non-negative function $\Phi$ of subcubes in $\Omega$ is  {\it subadditive} if 
$\sum\limits_{\omega'\in\square_\omega}\Phi(\omega')\leqslant\Phi(\omega)$ whenever $\square_\omega$ 
is a dyadic subdivision of a subcube $\omega$ of $\Omega$. 
For a cube $\omega\subset\Omega$, we set
\[
g_{\alpha}(\omega):=|\omega|^{\alpha}\Phi(\omega) \quad\text{ and }\quad N_\gamma(\omega):= \Big\lfloor \Big(\frac{|\Omega|}{|\omega|}\Big)^{\!\!\!\gamma\,}\Big\rfloor ,
\quad \alpha,\gamma\geqslant 0,
\]
and, for a dyadic subdivision $\square$ of $\Omega$, let
\[
G_{\alpha}(\square):=\max\limits_{\omega\in\square} g_{\alpha}(\omega).
\]

\begin{lemma}
	\label{lemma1}
	Let $\Omega$ be a cube in $\mathbb{R}^d$ and $\Phi$ be a non-negative subadditive function defined on subcubes in $\Omega$, and let $0\leqslant \gamma\leqslant \alpha$. 
	Assume that a sequence of dyadic subdivisions $\{\square_{k}\}_{k=0}^{\infty}$ of $\Omega$ is obtained recursively as follows: 
	set $\square_{0}:=\{\Omega\}$ and, for $k\in\mathbb{N}$, construct an elementary extension $\square_{k}$ of $\square_{k-1}$ by subdividing all cubes 
	$\omega\in\square_{k-1}$ satisfying
	\begin{equation}
		\label{subdivision}
		g_{\alpha}(\omega) \geqslant 2^{-d\alpha}G_{\alpha}\left(\square_{k-1}\right)
	\end{equation}
	into $2^d$ equal cubes. Then, for $k\in\mathbb{N}$, 
	\begin{equation}
		\label{estimate}
		G_{\alpha}\left(\square_{k}\right)\leqslant C_1(d,\gamma,\alpha)N_{k}^{-\frac{\alpha+1}{\gamma+1}}|\Omega|^\alpha\Phi(\Omega),
	\end{equation}
	where
	\begin{equation}\label{Nk}
		N_{k}:=\sum\limits_{\omega\in\square_k} N_\gamma(\omega),
	\end{equation}
	and the constant $C_1(d,\gamma,\alpha)$ depends only on $d$, $\gamma$ and $\alpha$.
\end{lemma}

\begin{remark}\rm
	Lemma~\ref{lemma1} in the case $\gamma = 0$ coincides with Theorem~2.1 in~\cite{BS67}.
\end{remark}

\begin{proof}
	Let $k\in\mathbb{N}$. Denote by $S_{k}$ the set of all cubes from $\square_{k-1}$ that are subdivided to obtain  $\square_k$, and set $t_k :=
	\sum\limits_{\omega\in S_k} N_\gamma(\omega)$. By the definition of $N_\gamma(\omega)$, we have a trivial estimate
	\[
	t_{k} \leqslant |\Omega|^\gamma\sum\limits_{\omega\in S_{k}}\left|\omega \right|^{-\gamma}.
	\]
	By the construction of $\square_k$ it is also clear that
	\[
	2^{-d\alpha}G_{\alpha}\left(\square_{k-1}\right)\leqslant\min\limits_{\omega\in S_{k}} g_\alpha(\omega).
	\]
	Combining the above inequalities and applying the definition of $g_\alpha$ we obtain
	\begin{align*}
		\left(2^{-d\alpha}G_{\alpha}(\square_{k-1})\right)^{\frac{\gamma+1}{\alpha+1}} & \leqslant \displaystyle \min\limits_{\omega\in S_{k}}(g_\alpha(\omega))^{\frac{\gamma+1}{\alpha+1}} = \min\limits_{\omega\in S_{k}}|\omega|^{\frac{\alpha(\gamma+1)}{\alpha+1}}\Phi^{\frac{\gamma+1}{\alpha+1}}(\omega) \\
		&\leqslant \Big(\sum\limits_{\omega\in S_{k}}N_\gamma(\omega)\Big)^{-1} \sum\limits_{\omega\in S_{k}} N_\gamma(\omega) |\omega|^{\frac{\alpha(\gamma+1)}{\alpha+1}}\Phi^{\frac{\gamma+1}{\alpha+1}}(\omega) \\
		&\leqslant  \displaystyle\frac{|\Omega|^\gamma}{t_{k}}\sum\limits_{\omega\in S_{k}}|\omega|^{\frac{\alpha-\gamma}{\alpha+1}} \Phi^{\frac{\gamma+1}{\alpha+1}}(\omega).
	\end{align*}
	Applying the H\"older inequality with parameters $\frac{\alpha+1}{\alpha-\gamma}$ and $\frac{\alpha+1}{\gamma+1}$, and using subadditivity of $\Phi$, we have
	\[
	\left(2^{-d\alpha}G_{\alpha}(\square_{k-1})\right)^{\frac{\gamma+1}{\alpha+1}} \leqslant \frac{|\Omega|^\gamma}{t_{k}} \Big(\sum\limits_{\omega\in S_{k}}|\omega|\Big)^{\frac{\alpha-\gamma}{\alpha+1}}\Big(\sum\limits_{\omega\in S_{k}}\Phi(\omega)\Big)^{\frac{\gamma+1}{\alpha+1}} \leqslant \frac{1}{t_{k}} |\Omega|^{\alpha\frac{\gamma+1}{\alpha+1}} \Phi^{\frac{\gamma+1}{\alpha+1}}(\Omega).
	\]
	The latter inequality can be rewritten as
	\begin{equation}
		\label{ineq for t}
		t_{k} \leqslant 2^{\frac{\alpha d (\gamma+1)}{\alpha+1}} \left(G_{\alpha}(\square_{k-1})\right)^{-\frac{\gamma+1}{\alpha+1}} |\Omega|^{\alpha\frac{\gamma+1}{\alpha+1}} \Phi^{\frac{\gamma+1}{\alpha+1}}(\Omega).
	\end{equation}
	Note that to obtain~\eqref{ineq for t} in the case $\alpha = \gamma$, instead of the H\"older inequality we only need to use non-negativity and subadditivity of function $\Phi$: $\sum\limits_{\omega\in S_k}\Phi(\omega) \leqslant \sum\limits_{\omega\in \square_{k-1}}\Phi(\omega) \leqslant \Phi(\Omega)$.
	
	Next, by the construction of dyadic subdivisions $\square_j$, $j\in\mathbb{N}$, we have
	\begin{equation}
		\label{ineq for G}
		\begin{array}{rcl}
			G_\alpha(\square_{j}) & \leqslant & \displaystyle \max\left\{2^{-d\alpha}G_\alpha(\square_{j-1});\max\limits_{\omega'\in\square_{j}\setminus\square_{j-1}} |\omega'|^{\alpha}\Phi(\omega')\right\}\\ 
			& \leqslant & \displaystyle\max\left\{2^{-d\alpha}G_\alpha(\square_{j-1});\max\limits_{\omega\in S_{j}} 2^{-d\alpha}|\omega|^{\alpha}\Phi(\omega)\right\} \leqslant 2^{-d\alpha} G_\alpha (\square_{j-1}).
		\end{array}
	\end{equation}
	Applying~\eqref{ineq for G} recursively we obtain that for every $k,j\in\mathbb{N}$, $j\leqslant k$,
	\begin{equation}
		\label{recurs ineq for G}
		G_\alpha (\square_{k-1}) \leqslant 2^{-d\alpha(k-j)} G_\alpha(\square_{j-1}).
	\end{equation}
	In addition, for every $j\in\mathbb{N}$, we have
	\begin{align*}
		N_j - N_{j-1} & =   \sum_{\omega\in S_{j}} \sum_{\tilde\omega\in\square_{j}:\, \tilde\omega\subset\omega} \left(N_\gamma\left(\tilde\omega\right)-N_\gamma\left(\omega\right)\right) \\
		& = \sum_{\omega\in S_{j}} \sum_{\tilde\omega\in \square_{j}:\, \tilde\omega\subset\omega} \Big(\Big\lfloor \Big(\frac{|\Omega|}{|\tilde\omega|}\Big)^{\!\!\!\gamma\,}\Big\rfloor  -\Big\lfloor \Big(\frac{|\Omega|}{|\omega|}\Big)^{\!\!\!\gamma\,}\Big\rfloor \Big)\\
		& =  \sum_{\omega\in S_{j}} \Big(2^d\Big\lfloor 2^{d\gamma}\Big(\frac{|\Omega|}{|\omega|}\Big)^{\!\!\!\gamma\,}\Big\rfloor  -\Big\lfloor \Big(\frac{|\Omega|}{|\omega|}\Big)^{\!\!\!\gamma\,}\Big\rfloor \Big)\\
		& \leqslant  \left(2^{d(\gamma+2)} -1\right)\sum_{\omega\in S_{j}}N_\gamma(\omega) = \left(2^{d(\gamma+2)} -1\right)t_j.
	\end{align*}
	Summing up the above inequalities for $j=1,\ldots,k$ and taking into account that $t_1 = N_0 = 1$, we obtain
	\[
	N_k - N_0 \leqslant \left(2^{d(\gamma+2)} -1\right) \sum_{j=1}^{k} t_j \leqslant 2^{d(\gamma+2)}\sum\limits_{j=1}^k t_j - N_0.
	\]
	Combining~\eqref{ineq for t},~\eqref{recurs ineq for G} and~\eqref{ineq for G} with the latter inequality, we conclude that
	\begin{align*}
		\displaystyle N_k & \leqslant  \displaystyle 2^{d(\gamma+2)}\sum_{j=1}^{k}t_j \leqslant 2^{d(\gamma+2)}\sum_{j=1}^{k} 2^{\frac{\alpha d(\gamma+1)}{\alpha+1}} \left(G_{\alpha}(\square_{j-1})\right)^{-\frac{\gamma+1}{\alpha+1}} |\Omega|^{\alpha\frac{\gamma+1}{\alpha+1}} \Phi^{\frac{\gamma+1}{\alpha+1}}(\Omega) \\
		& =  \displaystyle 2^{d(\gamma+2) + \frac{\alpha d (\gamma+1)}{\alpha+1}} |\Omega|^{\alpha\frac{\gamma+1}{\alpha+1}} \Phi^{\frac{\gamma+1}{\alpha+1}}(\Omega) \sum_{j=1}^{k} \left(G_{\alpha}(\square_{j-1})\right)^{-\frac{\gamma+1}{\alpha+1}} \\
		& \leqslant  \displaystyle 2^{d(\gamma+2) + \frac{\alpha d (\gamma+1)}{\alpha+1}} |\Omega|^{\alpha\frac{\gamma+1}{\alpha+1}} \Phi^{\frac{\gamma+1}{\alpha+1}}(\Omega) \left(G_{\alpha}(\square_{k-1})\right)^{-\frac{\gamma+1}{\alpha+1}} \sum_{j=1}^{k} 2^{-\frac{d\alpha (\gamma+1)}{\alpha+1}(k-j)} \\
		& \leqslant  \displaystyle 2^{d(\gamma+2)}\cdot\left(1-2^{-\frac{\alpha d (\gamma+1)}{\alpha+1}}\right)^{-1}   |\Omega|^{\alpha\frac{\gamma+1}{\alpha+1}} \Phi^{\frac{\gamma+1}{\alpha+1}}(\Omega) \left(G_{\alpha}(\square_{k})\right)^{-\frac{\gamma+1}{\alpha+1}},
	\end{align*}
	which finishes the proof of the lemma with 
	\[
	C_1(d,\gamma,\alpha) = \displaystyle 2^{\frac{d (\alpha+1)(\gamma+2)}{\gamma+1}}\left(1-2^{-\frac{\alpha d
			(\gamma+1)}{\alpha+1}}\right)^{-\frac{\alpha+1}{\gamma+1}}.\qed
	\]
\end{proof}

In addition, we will need a similar statement in the case when $0 < \alpha < \gamma$ and $\Phi^{\gamma/\alpha}$ rather than $\Phi$ is subadditive.

\begin{lemma}
	\label{lemma2}
	Let $\Omega$ be a cube in $\mathbb{R}^d$, let $0 < \alpha < \gamma$, and let $\Phi$ be a non-negative function defined on subcubes in $\Omega$ such that
	$\Phi^{\gamma/\alpha}$ is subadditive. Assume that a sequence of dyadic subdivisions $\{\square_{k}\}_{k=0}^{\infty}$ of $\Omega$ is constructed as in Lemma~\ref{lemma1}.
	Then, for $k\in\mathbb{N}$, 
	\[
	G_{\alpha}(\square_{k})\leqslant C_2(d,\gamma,\alpha) N_{k}^{-\frac{\alpha}{\gamma}} |\Omega|^{\alpha} \Phi(\Omega),
	\]
	where $N_k$ is given by \eqref{Nk},
	and the constant $C_2(d,\gamma,\alpha)$
	depends only on $d$, $\gamma$ and $\alpha$.
\end{lemma}

\begin{proof}
	Using the same notations $S_{k}$ and $t_k$ as in the proof of Lemma~\ref{lemma1}, we have 
	\[
	\left(1 - 2^{-d\gamma}\right)|\Omega|^\gamma\sum\limits_{\omega\in S_{k}}\left|\omega \right|^{-\gamma}
	\leqslant t_{k} \leqslant |\Omega|^\gamma\sum\limits_{\omega\in S_{k}}\left|\omega \right|^{-\gamma},\qquad k\in\mathbb{N},
	\]
	which is trivial when $k=1$ and in the case $k\geqslant 2$ follows from the fact that $|\Omega|\geqslant 2^{d}|\omega|$ for all $\omega\in S_k$. 
	Due to the construction of $\square_k$ it is clear that
	\[
	2^{-d\alpha}G_{\alpha}(\square_{k-1})\leqslant\min\limits_{\omega\in S_{k}} g_\alpha(\omega).
	\]
	Combining the above inequalities and applying the definition of $g_\alpha$ we infer
	\[
	2^{-d\alpha}G_{\alpha}(\square_{k-1}) \leqslant  \min\limits_{\omega\in S_{k}}g_\alpha(\omega) = \min\limits_{\omega\in S_{k}}|\omega|^{\alpha}\Phi(\omega) \leqslant \frac{|\Omega|^\gamma}{t_{k}}\sum\limits_{\omega\in S_{k}}|\omega|^{\alpha-\gamma}\Phi(\omega).
	\]
	Applying the H\"older inequality with parameters $\frac{\gamma}{\gamma-\alpha}$ and $\frac{\gamma}{\alpha}$, and using the subadditivity of $\Phi^{\frac{\gamma}{\alpha}}$ we obtain:
	\begin{align*}
		2^{-d\alpha}G_{\alpha}(\square_{k-1}) & \leqslant  \displaystyle\frac{|\Omega|^\gamma}{t_{k}} \Big(\sum\limits_{\omega\in S_{k}}|\omega|^{-\gamma}\Big)^{1-\frac{\alpha}{\gamma}} \Big(\sum\limits_{\omega\in S_{k}} \Phi^{\frac{\gamma}{\alpha}}(\omega)\Big)^{\frac{\alpha}{\gamma}} \\
		& \leqslant  \displaystyle\frac{|\Omega|^\alpha}{\left(1 - 2^{-d\gamma}\right)^{1 - \frac{\alpha}{\gamma}}t_{k}}\cdot  t_k^{1-\frac{\alpha}{\gamma}}\Phi(\Omega) = \frac{|\Omega|^\alpha\Phi(\Omega)}{\left(1 - 2^{-d\gamma}\right)^{1 - \frac{\alpha}{\gamma}}}\cdot t_{k}^{-\frac{\alpha}{\gamma}}.
	\end{align*}
	As a result, we have
	\begin{equation}
		\label{ineq for t 2}
		t_k \leqslant \frac{2^{d\gamma}|\Omega|^\gamma}{\left(1 - 2^{-d\gamma}\right)^{\frac{\gamma}{\alpha} - 1}} \left(G_\alpha(\square_{k-1})\right)^{-\frac{\gamma}{\alpha}} \Phi^{\frac{\gamma}{\alpha}} (\Omega).
	\end{equation}
	Repeating the arguments presented in the proof of Lemma~\ref{lemma1} with inequality~\eqref{ineq for t 2} being applied instead of~\eqref{ineq for t}, we obtain:
	\[
	N_k \leqslant 2^{2d(\gamma+1)}\cdot\left(1-2^{-d\gamma}\right)^{-\frac{\gamma}{\alpha}} |\Omega|^{\gamma} \Phi^{\frac{\gamma}{\alpha}}(\Omega) \left(G_{\alpha}(\square_{k-1})\right)^{-\frac{\gamma}{\alpha}}.
	\]
	Observing that inequality~\eqref{ineq for G} also holds true in the considered case and applying it to the latter inequality, we complete the proof of the lemma, with
	\[
	C_2(d,\gamma,\alpha) = 2^{\frac{d\alpha(\gamma+2)}{\gamma}}\left(1-2^{-d\gamma}\right)^{-1}.\qed
	\]
\end{proof}

\section{Upper estimates for $E_N(f)_p$}
\label{optorder}

In this section we establish the upper estimates of the error of the best nonlinear $L_p$-approximation of functions from Sobolev space $W_q^2(\Omega)$ by piecewise constants. 

We start with presenting an algorithm for constructing  a convex partition of $\Omega$ and corresponding
piecewise constant approximation for any function $f\in W_q^2(\Omega)$ whenever the conditions of embedding  into
$L_p(\Omega)$ are satisfied. To this end we will use the dyadic refinements defined in Lemma~\ref{lemma1} for $\gamma = \frac 1d$ and a specific
$\alpha$ depending on $d,p,q$. In this case each $N_\gamma(\omega) = \left(|\Omega|/|\omega|\right)^{1/d} $ is a non-negative
integer power of two,  
$N_{k+1} \leqslant 2^d N_{k}$ and the sequence $\left\{N_k\right\}_{k=0}^\infty$ is increasing.

\begin{algorithm}
	\label{algorithm1}
	Let $1 \leqslant p \leqslant \infty$ and $1 \leqslant q < \infty$ be such that 
	$\frac{2}{d} + \frac{1}{p} - \frac{1}{q} \geqslant 0$ when $p < \infty$ and 
	$\frac{2}{d} - \frac{1}{q} > 0$ when $p = \infty$, and let $f\in W_q^2(\Omega)$. 
	Given any $N\in\mathbb{N}$,  construct a sequence of dyadic subdivisions $\left\{\square_{k}\right\}_{k=0}^{m+1}$
	of $\Omega$ following the algorithm in Lemma~\ref{lemma1} with parameters: $\gamma = \frac 1d$, $\Phi(\cdot) := |f|^q_{W^1_q(\cdot)} + |f|^q_{W^2_q(\cdot)}$ and
	\begin{equation}\label{alpha}
		\alpha := q\left(\frac{2}{d}+\frac{1}{p}\left(1+\frac{1}{d}\right)-\frac{1}{q}\right), 
	\end{equation} 
	where $m$ is such that $N_{m}\leqslant N < N_{m+1}$. The convex partition $\triangle_N$ is obtained by subdividing each
	$\omega\in\square_m$ into $N_{\gamma}(\omega)$ slices by equidistant hyperplanes orthogonal to the average gradient ${\bf h}_\omega:=|\omega|^{-1}\int_{\omega}\nabla f({\bf
		x})\, {\rm d}{\bf x}$, and  the piecewise constant approximant $s_N(f)$ of $f$ is defined by
	\[
	s_N(f):= \sum_{\delta\in\triangle_N} f_\delta \cdot\chi_\delta,
	\]
	where 
	\[
	f_\delta:=\frac{1}{|\delta|} \int_\delta f({\bf x})\,{\rm d}{\bf x}, \qquad \chi_\delta({\bf x}) = \left\{\begin{array}{ll} 1, & {\bf x}\in\delta, \\ 0, & {\bf x}\in\Omega\setminus\delta. \end{array}\right.
	\]
\end{algorithm}

\medskip

A key tool for  estimating the error $\left\|f-s_N(f)\right\|_{L_p(\Omega)}$ is provided by Lemmas~\ref{lemma1} and~\ref{lemma2}. 
Observe that Lemma~\ref{lemma1} is applicable when $\alpha\geqslant\frac 1d$, which is equivalent to the inequality $\frac{2}{d+1} + \frac{1}{p} - \frac{1}{q}
\geqslant 0$. This leads to the optimal order estimate in Theorem~\ref{theorem1}. In the case $\alpha<\frac 1d$ we apply Lemma~\ref{lemma2} and obtain 
in Theorem~\ref{theorem2} certain suboptimal order
estimates that are nevertheless of interest whenever they improve the order $\mathcal{O}\big(N^{-\frac1d}\big)$ of isotropic piecewise constant approximation
\eqref{BSc}.

We will need the Sobolev-Poincar\'e inequality in the following form.
Let $\omega$ be a cube in $\mathbb{R}^d$, and let $1\leqslant\xi\leqslant\infty$ and $1\leqslant\eta<\infty$ be such that
either $\frac{1}{d} + \frac{1}{\xi} - \frac{1}{\eta} \geqslant 0$ when $\xi<\infty$, or 
$\frac{1}{d} - \frac{1}{\eta} > 0$ when $\xi=\infty$.
Then for every $g\in W^1_\eta\left(\omega\right)$ with zero mean value over $\omega$, there exists a constant $C_\mathrm{SP}(d,\xi,\eta) > 0$ such that
\begin{equation}\label{SP}
	\|g\|_{L_\xi\left(\omega\right)}\leqslant C_\mathrm{SP}(d,\xi,\eta)\,|\omega|^{\frac{1}{d}+\frac{1}{\xi}-\frac{1}{\eta}}\left\|\nabla g\right\|_{L_\eta\left(\omega\right)},
\end{equation}
where, for  any vector-function $f = \left(f_1,\ldots,f_d\right):\omega\to\mathbb{R}^d$ with components $f_j\in L_\xi(\omega)$, $j=1,\ldots,d$, we set
\[
\|f\|_{L_\xi(\omega)} := \big\|(f_1^2 + \ldots + f_d^2)^{1/2}\big\|_{L_\xi(\omega)},\quad 1\leqslant\xi\leqslant\infty.
\]
A proof of \eqref{SP} for the case of the unit cube can be found e.g.\ in \cite[Theorem 8.12]{LiebLoss_01}, and the general case follows by using an affine mapping between $[0,1]^d$
and $\omega$.

\begin{theorem}
	\label{theorem1}
	Let $d\geqslant 2$, $N\in\mathbb{N}$, and let $1 \leqslant p\leqslant \infty$ and $1 \leqslant q< \infty$ be such that
	\begin{equation}
		\label{qcond}
		\frac{2}{d+1} + \frac{1}{p} - \frac{1}{q} \geqslant 0.
	\end{equation}
	Then for any $f\in W^2_q(\Omega)$ the error of the piecewise constant approximant 
	$s_N(f)$ generated by Algorithm~\ref{algorithm1} satisfies 
	\[
	E_N(f)_p\leqslant \left\|f-s_N(f)\right\|_{L_p(\Omega)}\leqslant C_3(d,p,q,|\Omega|)\, N^{-\frac{2}{d+1}} \left( |f|^q_{W^1_q(\Omega)} + |f|^q_{W^2_q(\Omega)} \right)^{\frac{1}{q}},
	\]
	where the constant $C_3(d,p,q,|\Omega|)$ depends only on $d,p,q$ and $|\Omega|$.
\end{theorem}

\begin{remark}\rm
	The assertion of Theorem~\ref{theorem1} in the case $q=p$ was proved in \cite{D12} and in the case $q > p$ easily follows from that result.
\end{remark}

\begin{proof} 
	For every cube $\omega\in\square_m$, we denote by ${\bf c}_\omega$ the center of $\omega$ and consider the linear approximant of $f$ on $\omega$
	in the form (see~\cite{D14}):
	\[
	\ell_{\omega}({\bf x}):=\displaystyle \frac{1}{|\omega|} \int_{\omega} f({\bf x})\,{\rm d}{\bf x} + \left<{\bf h}_\omega,\;{\bf x} - {\bf c}_{\omega}\right>,
	\] 
	where $\left<\cdot, \cdot\right>$ is the scalar product in $\mathbb{R}^d$. 
		
	We start with estimating the $L_p$-error of approximation of $f$ by the piecewise constant function 
	\[
	\tilde{s}_N(\ell) := \sum_{\delta\in\triangle}\tilde{s}_\delta\cdot\chi_\delta,
	\] 
	where 
	$$
	\tilde{s}_\delta:= |\delta|^{-1} \int_\delta \ell_{\omega}({\bf x})\,{\rm d}{\bf x},\quad 
	\delta \in \Lambda_{\omega} :=  \{\delta\in\triangle_N\,:\, \delta\subset\omega\}.$$ 	
	In view of triangle inequality, for every $\omega\in\square_m$,
	\begin{equation*}
		\left\|f - \tilde{s}_N(\ell) \right\|_{L_p \left(\omega\right)} \leqslant  \left\| f - \ell_{\omega} \right\|_{L_p \left(\omega\right)} + \left\| \ell_{\omega} - \tilde{s}_N(\ell) \right\|_{L_p \left(\omega\right)}.
	\end{equation*}
	
	Let us estimate $\|f-\ell_\omega\|_{L_p(\omega)}$ by applying the Sobolev-Poincar\'e inequality \eqref{SP} twice. To this end we choose $1\leqslant \tau\leqslant\infty$
	such that simultaneously $\frac{1}{d} + \frac{1}{p} -\frac{1}{\tau} \geqslant 0$ and $\frac{1}{d} + \frac{1}{\tau} -\frac{1}{q} \geqslant 0$, for example 
	$\tau=\frac{2pq}{p+q}$. Since $\int_\omega \left(f({\bf x}) - \ell_\omega({\bf x})\right)\,{\rm d}{\bf x} = 0$, we can apply \eqref{SP}:
	\[
	\displaystyle \left\|f - \ell_{\omega} \right\|_{L_p(\omega)} \leqslant k_1  |\omega|^{\frac{1}{d} + \frac{1}{p} -\frac{1}{\tau}} \left\|\nabla f -  {\bf h}_\omega\right\|_{L_\tau(\omega)},
	\]
	where $k_1=C_\mathrm{SP}(d,p,\tau)$. Using the triangle inequality and the Sobolev-Poincar\'e inequality for the second time, we obtain
	\[
	\begin{array}{rcl}
	\displaystyle \|\nabla f - {\bf h}_\omega\|_{L_\tau(\omega)} 
	& = & \displaystyle \Big\|\Big( \sum\limits_{j=1}^d \Big( \frac{\partial f}{\partial x_j} 
	- \frac{1}{|\omega|} \int_{\omega} \frac{\partial f}{\partial x_j} ({\bf x})\, {\rm d}{\bf x} \Big)^2 \Big)^{\frac{1}{2}} \Big\|_{L_\tau(\omega)}\\
	&\leqslant & \displaystyle \sum\limits_{j=1}^d \Big\|\frac{\partial f}{\partial x_j} 
	- \frac{1}{|\omega|} \int_{\omega} \frac{\partial f}{\partial x_j} ({\bf x})\, {\rm d}{\bf x} \Big\|_{L_\tau(\omega)}\\
	&\leqslant & \displaystyle \sum\limits_{j=1}^d k_2 |\omega|^{\frac{1}{d} + \frac{1}{\tau} -\frac{1}{q}} \Big\| \nabla \Big(\frac{\partial f}{\partial x_j} \Big)\Big\|_{L_q(\omega)} 
	\leqslant \displaystyle k_2 |\omega|^{\frac{1}{d} + \frac{1}{\tau} -\frac{1}{q}} |f|_{W_q^2(\omega)},
	\end{array}
	\]
	where $k_2=C_\mathrm{SP}(d,\tau,q)$. Combining the last two inequalities we obtain
	\begin{equation}
		\label{approx_func_by_linear}
		\left\|f - \ell_\omega\right\|_{L_\tau(\omega)} \leqslant k_3 |\omega|^{\frac{2}{d} + \frac{1}{p} - \frac{1}{q}} |f|_{W^2_q(\omega)},
	\end{equation}
	where $k_3:= k_1 k_2$ depends only on $d$, $p$, $\tau$ and $q$.
	
	Now we estimate the $L_p$-distance between $\ell_\omega$ and $\tilde{s}_N(\ell)$ on $\omega$,
	\begin{equation}
		\label{const_aprox}
		\begin{array}{rcl}
			\displaystyle \left\| \ell_\omega - \tilde{s}_N(\ell) \right\|^p_{L_p \left(\omega\right)} 
			\!& \!= \!& \!\displaystyle \sum\limits_{\delta\in\Lambda_\omega} \left\| \ell_\omega - \tilde{s}_{\delta}\right\|^p_{L_p \left(\delta\right)} 
			= \sum\limits_{\delta\in\Lambda_\omega} \int_{\delta} \Big| \Big< {\bf h}_{\omega},\;{\bf x} - \frac{1}{|\delta|}  \int_{\delta} {\bf u}\, {\rm d}{\bf u} \Big>\Big|^p {\rm d}{\bf x}\\
			\!& \!= \!& \!\displaystyle \sum\limits_{\delta\in\Lambda_\omega}\frac{1}{|\delta|^p}  \int_{\delta} \Big| \int_{\delta} \left< {\bf h}_{\omega}, {\bf x} - {\bf u}\right> {\rm d}{\bf u} \Big|^p {\rm d}{\bf x}\\
			\!& \!\leqslant\!& \!\displaystyle \sum\limits_{\delta\in\Lambda_\omega} \frac{1}{|\delta|^p}  
			\int_{\delta} \Big| \int_{\delta} \left\|{\bf h}_{\omega}\right\|_2 ({\bf x}-{\bf u})_{{\bf h}_{\omega}}\, {\rm d}{\bf u} \Big|^p {\rm d}{\bf x},
		\end{array}
	\end{equation}
	where $({\bf x}-{\bf u})_{{\bf h}_{\omega}}$ denotes the length of the projection of ${\bf x} - {\bf u}$ on a unit vector parallel to ${\bf h}_{\omega}$ and $\left\|{\bf h}_\omega\right\|_{2}$ stands for the Euclidean norm of ${\bf h}_\omega$. 	
	By the H\"older inequality,
	\[
	\begin{array}{rcl}
	\displaystyle\left\|{\bf h}_{\omega}\right\|_2 
	& = &\displaystyle \Big(\sum\limits_{k=1}^d \Big(\frac{1}{|\omega|} \int_{\omega} \frac{\partial f}{ \partial x_k} ({\bf x})\,{\rm d}{\bf x}\Big)^2\Big)^{\frac12} 
	\leqslant \displaystyle \frac{1}{|\omega|}\sum\limits_{k=1}^d \int_{\omega} \Big|\frac{\partial f}{ \partial x_k} ({\bf x})\Big|\,{\rm d}{\bf x} \\
	& \leqslant & \displaystyle \frac{1}{|\omega|} \sum\limits_{k=1}^d |\omega|^{1-\frac{1}{q}}\Big(\int_{\omega} \Big|\frac{\partial f}{ \partial x_k}({\bf x})\Big|^q {\rm d}{\bf x}\Big)^{\frac{1}{q}}
	= \displaystyle|\omega|^{-\frac{1}{q}} |f|_{W^1_q(\omega)}.
	\end{array}
	\]
	It is also clear that, for every $\delta\in\Lambda_\omega$ and ${\bf x},{\bf u} \in \delta$, we have $({\bf x}-{\bf u})_{{\bf h}_{\omega}}\leqslant\frac{\textrm{diam}\,(\omega)}{N_\gamma(\omega)}$.
	Substituting the above inequalities into~\eqref{const_aprox} we obtain
	\begin{equation}
		\label{approx_linear_by_constant}
		\begin{array}{rcl}
			\displaystyle\left\| \ell_\omega - \tilde{s}_N(\ell)\right\|^p_{L_p \left(\omega\right)} & \leqslant & \displaystyle \sum\limits_{\delta\in\Lambda_\omega} |\delta|  \cdot|\omega|^{-\frac{p}{q}} 
			\displaystyle \Big(\frac{\textrm{diam}\,(\omega)}{N_{\gamma}(\omega)}\Big)^p |f|^p_{W^1_q(\omega)}\\
			& = & \displaystyle d^{\frac p2}\cdot    |\omega|^{1+\frac{2p}{d}-\frac{p}{q}} |\Omega|^{-\frac pd} |f|^p_{W^1_q(\omega)}.
		\end{array}
	\end{equation}
	
	Combining \eqref{approx_func_by_linear} and~\eqref{approx_linear_by_constant} and applying the H\"older inequality we obtain
	\begin{equation}
		\label{ineq_with_w_and_norm}
		\begin{array}{rcl}
			\displaystyle \left\| f - \tilde{s}_N \left(\ell\right) \right\|_{L_p\left(\omega\right)} & \leqslant & \displaystyle k_4\cdot \left|\omega\right|^{\frac{2}{d} + \frac{1}{p} - \frac{1}{q}} \left( \left|f\right|_{W^1_q\left(\omega\right)} + \left|f\right|_{W^2_q(\omega)} \right) \\
			& \leqslant & \displaystyle k_5\cdot \left|\omega\right|^{\frac{2}{d} + \frac{1}{p} - \frac{1}{q}} \left( \left|f\right|^q_{W^1_q\left(\omega\right)} + \left|f\right|^q_{W^2_q(\omega)} \right)^{\frac 1q},
		\end{array}
	\end{equation}
	where $k_4:= \max\left\{k_3;\sqrt{d} \, |\Omega|^{-\frac{1}{d}} \right\}$ and $k_5 := 2^{1-\frac 1q}k_4$. Using definitions of $\alpha$ and functions $g_\alpha$ and
	$G_\alpha$ we rewrite \eqref{ineq_with_w_and_norm} in the following way 
	\[
	\left\| f - \tilde{s}_N \left(\ell\right) \right\|^p_{L_p\left(\omega\right)} \leqslant k_5^p\cdot \left| \omega\right|^{-\frac{1}{d}} \left(g_\alpha(\omega)\right)^{\frac{p}{q}} \leqslant k_5^p \left| \omega\right|^{-\frac{1}{d}} \left(G_\alpha\left(\square_m\right)\right)^{\frac{p}{q}}.
	\]
	Hence,
	\[
	\left\| f - \tilde{s}_N \left(\ell\right)\right\|^p_{L_p \left(\Omega\right)} = \sum_{\omega \in \square_m}\left\| f - \tilde{s}_N \left(\ell\right) \right\|^p_{L_p\left(\omega\right)} \leqslant k_5^p \sum_{\omega \in \square_m} \left|\omega\right|^{-\frac{1}{d}} G_\alpha^{\frac{p}{q}} \left(\square_m\right) = k_5^p\cdot|\Omega|^{-\frac 1d}\cdot N_m \cdot G_\alpha^{\frac{p}{q}} \left(\square_m\right).
	\]
	
	Combining this inequality with Lemma~\ref{lemma1}, we arrive at
	\[
	\begin{array}{rcl}
	\displaystyle \left\|f - \tilde{s}_N\left(\ell\right)\right\|_{L_p \left(\Omega\right)} 
	&\leqslant & \displaystyle k_5\cdot|\Omega|^{-\frac 1{pd}} N_m^{\frac{1}{p}} G_\alpha^{\frac{1}{q}} \left(\square_m\right) \leqslant \displaystyle k_6 N_m^{\frac{1}{p}} N_m^{-(\alpha+1)\frac{1}{q(1+\frac{1}{d})}} {\Phi^{\frac{1}{q}}(\Omega)} \\
	& \leqslant & \displaystyle k_7  N^{ -\frac{2}{d+1}} \left( |f|^q_{W^1_q(\Omega)} + |f|^q_{W^2_q(\Omega)} \right)^{\frac{1}{q}},
	\end{array}
	\]
	where $k_6:= k_5 C_1^{\frac 1q}\left(d,\gamma,\alpha\right)\cdot |\Omega|^{\frac{\alpha}{q}-\frac{1}{pd}}$ and $k_7 := 2^{-\frac{2d}{d+1}}k_6$, with $C_1$ from \eqref{estimate} and $\gamma,\alpha$ as in Algorithm~\ref{algorithm1}.
	
	Finally, for every $\delta \in \triangle_N$, it is easy to see that 
	$\left\|f - f_\delta\right\|_{L_p(\delta)} \leqslant 2\inf\limits_{c\in\mathbb{R}} \|f - c\|_{L_p(\delta)}$, 
	where $f_\delta$ was defined in Algorithm~\ref{algorithm1}. Therefore
	\[
	\left\|f - s_N\left(f\right)\right\|_{L_p\left(\Omega\right)} \leqslant 2 k_7 N^{ -\frac{2}{d+1}} 
	\left( |f|^q_{W^1_q(\Omega)} + |f|^q_{W^2_q(\Omega)} \right)^{\frac{1}{q}},
	\]
	which proves the theorem with
	\[
	C_3(d,p,q,|\Omega|) = 2^{\frac{2}{d+1}-\frac 1q}\, C_1^{\frac 1q}(d,\gamma,\alpha)\left|\Omega\right|^{\frac{\alpha}{q} - \frac{1}{pd}}\, \max{\left\{C_\mathrm{SP}(d,p,\tau) C_\mathrm{SP}(d,\tau,q);\sqrt{d}\,|\Omega|^{-\frac 1d}\right\}}.\qed
	\]
\end{proof}

We now establish an upper estimate for the order of $\left\| f - s_N \left( f\right) \right\|_{L_p \left(\Omega\right)}$ as $N\to\infty$ 
provided that $q$ does not satisfy~\eqref{qcond}. To this end we will apply Lemma~\ref{lemma2} instead of Lemma~\ref{lemma1}.

\begin{theorem}
	\label{theorem2}
	Let $d\geqslant2$, $N\in\mathbb{N}$, and let $1 \leqslant p\leqslant \infty$ and $1\leqslant q<\infty$ be such that
	\[
	\frac{2}{d+1} + \frac{1}{p} -\frac{1}{q} < 0\quad\textrm{and}\quad \left\{\begin{array}{ll} \displaystyle \frac{2}{d} + \frac{1}{p} - \frac{1}{q} \geqslant 0, & p < \infty,\\ [10pt] \displaystyle \frac{2}{d} + \frac{1}{p} - \frac{1}{q} > 0, & p = \infty.\end{array}\right.
	\]
	Then for any  $f\in W_q^2(\Omega)$ the piecewise constant spline $s_N(f)$ generated by Algorithm~\ref{algorithm1} satisfies 
	\[
	E_N(f)_p \leqslant \left\| f - s_N \left(f\right) \right\|_{L_p\left(\Omega\right)} \leqslant C_4(d,p,q,|\Omega|) N^{-d\left(\frac{2}{d} + \frac{1}{p} - \frac{1}{q}\right)}
	\left(|f|^q_{W^1_q(\Omega)} + |f|^q_{W^2_q(\Omega)}\right)^{\frac{1}{q}},
	\]
	where the constant $C_4(d,p,q,|\Omega|)$ depends only on $d,p,q,|\Omega|$.
	In particular, 
	\[
	E_N(f)_p\leqslant \left\| f - s_N \left(f\right) \right\|_{L_p \left(\Omega\right)}=\smallo(N^{-\frac1d}),\qquad f\in W_q^2(\Omega),
	\]
	as soon as
	\begin{equation}
		\label{nqcond}
		\frac{2}{d}  + \frac{1}{p} - \frac{1}{q} > \frac{1}{d^2}.
	\end{equation}
\end{theorem}

\begin{proof}
	Repeating the proof of Theorem~\ref{theorem1} and applying Lemma~\ref{lemma2} we obtain the estimate
	\[
	\displaystyle \left\|f - \tilde{s}_N\left(\ell\right)\right\|_{L_p \left(\Omega\right)} 
	\leqslant \displaystyle k_5|\Omega|^{-\frac{1}{pd}} N_m^{\frac 1p} G_\alpha^{\frac{1}{q}}\left(\square_m\right) 
	\leqslant k_6N_m^{\frac{1}{p}} N_m^{-\frac{d\alpha}{q}} \Phi^{\frac{1}{q}}(\Omega) 
	\leqslant \displaystyle k_7N^{-d\left(\frac{2}{d} + \frac{1}{p} - \frac{1}{q}\right)} \Phi^{\frac{1}{q}}(\Omega),
	\]
	where $k_6 = k_6(d,p,\tau,q,|\Omega|) := k_5C_2^{\frac 1q}\left(d,\gamma,\alpha\right)\left|\Omega\right|^{\frac{\alpha}{q}-\frac{1}{pd}}$ and $k_7 := 2^{-\frac{2d}{d+1}}k_6$, with the constant $C_2\left(d,\gamma,\alpha\right)$  of Lemma~\ref{lemma2}. The proof is completed by the same arguments as in the proof of Theorem~\ref{theorem1}. The estimate is valid with
	\[
	C_4(d,p,q,|\Omega|) = 2^{\frac{2}{d+1}-\frac 1q} C_2^{\frac 1q}\left(d,\gamma,\alpha\right)\left|\Omega\right|^{\frac{\alpha}{q} - \frac{1}{pd}}
	\max{\left\{C_\mathrm{SP}(d,p,\tau)C_\mathrm{SP}(d,\tau,q);\sqrt{d}|\Omega|^{-\frac 1d}\right\}},
	\]
	with $\gamma$ and $\alpha$ as in Algorithm~\ref{algorithm1} and $1\leqslant\tau\leqslant\infty$ being any number such that $\frac{1}{d} + \frac{1}{p} - \frac{1}{\tau}\geqslant 0$ and $\frac{1}{d} + \frac{1}{\tau} - \frac{1}{q} \geqslant 0$. $\qed$
\end{proof}

Using the embedding theorem $W^r_q(\Omega)\subset W^2_{q'}(\Omega)$, where $\frac 1{q'} =  \frac{r-2}{d}+\frac 1q$ \cite[p. 328]{Triebel_78}, 
we deduce from Theorems~\ref{theorem1} and~\ref{theorem2} the following statement that summarizes the main results of this section.

\begin{corollary}
	\label{cor}
	Let $d\geqslant 2$, $r\geqslant 2$, $1\leqslant p\leqslant \infty$, $1 \leqslant q<\infty$. Then
	\begin{align}
		\label{cor1}
		E_N(W^r_q(\Omega))_p
		&=\mathcal{O}(N^{-\frac{2}{d+1}})\quad\text{if}\qquad \frac{r}{d}  + \frac{1}{p} - \frac{1}{q} \geqslant \frac{2}{d(d+1)},\\
		\label{cor2}
		E_N(W^r_q(\Omega))_p
		&=\smallo(N^{-\frac1d})\qquad\text{if}\qquad \frac{r}{d}  + \frac{1}{p} - \frac{1}{q} > \frac{1}{d^2},
	\end{align}
	where 
	$$
	E_N(W^r_q(\Omega))_p:=\sup\limits_{f\in W^r_q(\Omega),\atop \left\|f\right\|_{W^r_q(\Omega)} \leqslant 1}E_N(f)_p.$$
\end{corollary}

\section{An estimate from below}
\label{below}

As the above results only apply to $W^{r}_q(\Omega)$ with $r\geqslant 2$, there remains the question whether the orders $E_N(f)_p=\mathcal{O}(N^{-\frac{2}{d+1}})$
or at least $E_N(f)_p=\smallo(N^{-\frac1d})$ can be achieved when $r<2$. Leaving out the question of the validity of these bounds for individual functions in
$W^{r}_q(\Omega)$, we provide an estimate from below for $E_N(W^r_q(\Omega))_\infty$ that shows in particular that the optimal order of \eqref{cor1} is not achieved
when $p=\infty$ and $r < \frac{2d}{d+1}$.

\begin{theorem}
	\label{lb}
	Let $d\geqslant 2$, $r\in(0,2)$, $1\leqslant q <\infty$, and $\frac{r}{d} - \frac{1}{q} > 0$. Then
	\begin{equation}
		\label{lbeq}
		\overline{\lim\limits_{N \to \infty}} \sup\limits_{f\in W^r_q(\Omega),\atop \left\|f\right\|_{W^r_q(\Omega)} \leqslant 1} N^{\frac{r}{d}}E_N(f)_\infty> 0.
	\end{equation}
	In particular,	$E_N(W^r_q(\Omega))_\infty\ne \smallo(N^{-\frac{r}{d}})$ and
	\begin{equation}
		\label{lbeq2}
		E_N(W^r_q(\Omega))_\infty\ne \mathcal{O}(N^{-\frac{2}{d+1}})\qquad\text{if}\qquad r < \frac{2d}{d+1}.
	\end{equation}
\end{theorem}

\begin{proof}
	For simplicity, assume that $\Omega = (0,1)^d$. For $m\in\mathbb{N}$, let 
	$\square_m = \left\{\omega_{\bf i}\right\}_{{\bf i}\in\mathcal{N}_m^d}$ be the partition of $\Omega$ into $m^d$ subcubes
	\[
	\textstyle\omega_{\bf i} := \left(\frac{i_1-1}{m},\frac{i_1}{m}\right)\times\cdots\times\left(\frac{i_d-1}{m},\frac{i_d}{m}\right),
	\qquad {\bf i}=(i_1,\ldots,i_d)\in\mathcal{N}_m^d,
	\]
	where 
	\[
	\mathcal{N}_m^d := \left\{{\bf i}\in\mathbb{N}^d\,:\,1\leqslant i_k\leqslant m,\;k=1,\ldots,d\right\}.
	\] 	
	
	We set ${\bf 1} = (1,\ldots,1)\in\mathbb{N}^d$ and $d({\bf i},{\bf j}) := \max\limits_{k=1,\dots,d} |i_k-j_k|$, ${\bf i},{\bf j}\in\mathbb{N}^d$, 
	and consider the infinitely differentiable compactly supported function $\varphi:\mathbb{R}^d\to\mathbb{R}$ defined by
	\[
	\varphi({\bf x}) := \left\{\begin{array}{ll}
	\exp{\big(-\frac{1}{1-|2{\bf x}-{\bf 1}|^2}\big)}, & |2{\bf x}-{\bf 1}|<1, \\
	0, & |2{\bf x}-{\bf 1}|\geqslant 1,
	\end{array} \right.
	\]
	where $|\cdot|$ stands for the Euclidean norm in $\mathbb{R}^d$. For ${\bf i}\in\mathcal{N}_m^d$, we scale $\varphi$ into the cube $\omega_{\bf i}$ as
	$\varphi_{\bf i}({\bf x}):=\varphi(m{\bf x}-{\bf i}+{\bf 1})$, such that $\varphi_{\bf i}({\bf x})\ne0$ on the set
	$\{{\bf x}\in\Omega:\,|{\bf x}-\frac{2{\bf i}-{\bf 1}}{2m}|<\frac{1}{2m}\}\subset\omega_{\bf i}$.
	Then
	\begin{equation}
		\label{seminorm_varphi_i}
		\left|\varphi_{\bf i}\right|_{W^s_q\left(\omega_{\bf i}\right)} = m^{s - \frac{d}{q}} \left|\varphi\right|_{W_q^s(\Omega)},\quad s\geqslant0.
	\end{equation}
	
	We construct a function $f_m\in W^r_q(\Omega)$ as  $f_m := \sum\limits_{{\bf i}\in \mathcal{N}_m^d} \varphi_{\bf i}$, and claim that
	\begin{equation}
		\label{norm_fm}
		\|f_m\|_{W^r_q(\Omega)}\le C_5 m^r \left\|\varphi\right\|_{W^r_q(\Omega)},
	\end{equation}
	where $C_5$ is a positive constant independent of $m$. By \eqref{seminorm_varphi_i} we see that for any $k\in\mathbb{Z}_+$, 
	$\left|f_m\right|_{W^k_q(\Omega)}^q 
	= m^d\cdot  m^{kq-d}\left|\varphi\right|_{W^k_q(\Omega)}^q 
	= m^{kq} \left|\varphi\right|_{W^k_q(\Omega)}^q$, which implies
	\begin{equation}
		\label{norm_fmk}
		\|f_m\|_{W^{\lfloor r\rfloor}_q(\Omega)}\le m^{\lfloor r\rfloor} \left\|\varphi\right\|_{W^r_q(\Omega)},
	\end{equation}
	and already shows \eqref{norm_fm} in the case $r=1$.
	
	It remains to give  an upper estimate for the seminorm of $f_m$ in the Sobolev-Slobodeckij space $W^r_q(\Omega)$ when $r>0$ and $r\ne 1$. We have
	\begin{equation}\label{fullsum}
		|f_m|^q_{W^r_q(\Omega)} 	\leqslant \sum\limits_{{\bf i}\in \mathcal{N}_m^d} \sum\limits_{{\bf j}\in\mathcal{N}_m^d} \sup\limits_{{\bf k}\in \mathbb{Z}_+^d\atop |{\bf k}|=\lfloor r\rfloor }  
		\int_{\omega_{\bf i}} \int_{\omega_{\bf j}} \frac{\left|D^{{\bf k}} \varphi_{\bf i}({\bf x}) - D^{{\bf k}} \varphi_{\bf j}({\bf y})\right|^q}{\left| {\bf x}-{\bf y}\right|^{\left\{r\right\}q+d}}\,{\rm d}{\bf x}{\rm d}{\bf y}.
	\end{equation}
	We split the sum \eqref{fullsum} into three pieces corresponding to ${\bf i},{\bf j}$ satisfying $d({\bf i},{\bf j})=0$, $d({\bf i},{\bf j})=1$ and 
	$d({\bf i},{\bf j})\geqslant 2$, respectively, and estimate each of them separately.
	
	For the case $d({\bf i},{\bf j})=0$, that is ${\bf i}={\bf j}$, we have the following equality
	\begin{equation}
		\label{estimatesum1}
		\sum\limits_{{\bf i}\in\mathcal{N}_m^d} \sup\limits_{{\bf k}\in \mathbb{Z}_+^d\atop |{\bf k}|
			=\lfloor r\rfloor } \int_{\omega_{\bf i}} \int_{\omega_{\bf i}} \frac{\left|D^{{\bf k}} \varphi_{\bf i}({\bf x}) - D^{{\bf k}} \varphi_{\bf i}({\bf y})\right|^q}{\left| {\bf x}-{\bf y}\right|^{\left\{r\right\}q+d}}\,{\rm d}{\bf x}{\rm d}{\bf y} 
		= \sum\limits_{{\bf i}\in\mathcal{N}_m^d} |\varphi_{\bf i}|^q_{W^r_q(\omega_{\bf i})} = m^{rq} \left|\varphi\right|_{W_q^r(\Omega)}^q.
	\end{equation}		
	
	If $d({\bf i},{\bf j})\geqslant 2$, then we use the substitution 
	${\bf y} = {\bf u} + \frac{1}{m}\,{\bf c}^{{\bf i},{\bf j}}$, with ${\bf c}^{{\bf i},{\bf j}} := {\bf j} - {\bf i}$, to obtain
	$$
	\int_{\omega_{\bf i}} \int_{\omega_{\bf j}} \frac{\left|D^{{\bf k}} \varphi_{\bf i}({\bf x}) - D^{{\bf k}} \varphi_{\bf j}({\bf y})\right|^q}{\left| {\bf x}-{\bf y}\right|^{\left\{r\right\}q+d}}\,{\rm d}{\bf x}{\rm d}{\bf y}
	=\int_{\omega_{\bf i}} \int_{\omega_{\bf i}} \frac{\left|D^{{\bf k}} \varphi_{\bf i}({\bf x}) - D^{{\bf k}} \varphi_{\bf i}({\bf u})\right|^q}{\left| {\bf x}-{\bf u} -\frac{1}{m} {\bf c}^{{\bf i}, {\bf j}}\right|^{\left\{r\right\}q+d}}\,{\rm d}{\bf x}{\rm d}{\bf u}.$$		
	Since		$d({\bf i},{\bf j}) - 1\geqslant\frac12 d({\bf i},{\bf j})$ and $|x_k - u_k| \leqslant \frac{1}{m}$, $k=1,\ldots,d$,
	as long as ${\bf x}, {\bf u} \in \omega_{\bf i}$, we have
	$$
	\big|{\bf x}-{\bf u} - \frac{1}{m} \,{\bf c}^{{\bf i}, {\bf j}}\big|\geqslant \frac{1}{m}\left(d({\bf i},{\bf j}) - 1\right) 
	\geqslant\frac1{2m} d({\bf i},{\bf j}),$$
	and hence
	\begin{align*}
		\int_{\omega_{\bf i}} \int_{\omega_{\bf i}} \frac{\left|D^{{\bf k}} \varphi_{\bf i}({\bf x}) - D^{{\bf k}} \varphi_{\bf i}({\bf u})\right|^q}{\left| {\bf x}-{\bf u} - \frac{1}{m} {\bf c}^{{\bf i}, {\bf j}}\right|^{\left\{r\right\}q+d}}\,{\rm d}{\bf x}{\rm d}{\bf u}
		&\leqslant \displaystyle \big(\tfrac{2m}{d({\bf i},{\bf j})}\big)^{\{r\}q+d} \int_{\omega_{\bf i}} \int_{\omega_{\bf i}} \left|D^{{\bf k}} \varphi_{\bf i}({\bf x}) - D^{{\bf k}} \varphi_{\bf i}({\bf u})\right|^q\, {\rm d}{\bf x}{\rm d}{\bf u}
		\\ &\leqslant 2^{q+\{r\}q+d} m^{\{r\}q} \left|\varphi_{\bf i}\right|^q_{W^{\lfloor r\rfloor }_q(\omega_{\bf i})} d({\bf i},{\bf j})^{-\{r\}q-d}.
	\end{align*}
	Now
	$$
	\sum\limits_{{\bf i},{\bf j}\in \mathcal{N}_m^d\atop d({\bf i},{\bf j})\geqslant 2}
	\frac{1}{d({\bf i},{\bf j})^{\{r\}q+d}} 
	\leqslant m^d\sum\limits_{{\bf i}\in \mathcal{N}_m^d\atop d({\bf 0},{\bf i})\geqslant 2}\frac{1}{d({\bf 0},{\bf i})^{\{r\}q+d}} 
	\leqslant m^d\cdot \int_{|{\bf x}|\geqslant 1} \frac{{\rm d}{\bf x}}{|{\bf x}|^{\{r\}q+d}} = \frac{\mu_d m^d}{\{r\}q},$$
	where $\mu_d$ is the volume of the $(d-1)$-dimensional unit sphere. (Remark that $\{r\}\ne 0$ as $r\not\in\mathbb{N}$.)
	It follows that
	\begin{equation}
		\label{estimatesum3}
		\sum\limits_{{\bf i},{\bf j}\in \mathcal{N}_m^d\atop d({\bf i},{\bf j})\geqslant 2} 
		\sup\limits_{{\bf k}\in \mathbb{Z}_+^d\atop |{\bf k}|=\lfloor r\rfloor } 
		\int_{\omega_{\bf i}} \int_{\omega_{\bf j}} \frac{\left|D^{{\bf k}} \varphi_{\bf i}({\bf x}) - D^{{\bf k}} \varphi_{\bf j}({\bf y})\right|^q}{\left| {\bf x}-{\bf y} \right|^{\left\{r\right\}q+d}}\,{\rm d}{\bf x}{\rm d}{\bf y}
		\leqslant 
		\frac{2^{q+\{r\}q+d}\mu_dm^{rq}}{\{r\}q} |\varphi|^q_{W^{\lfloor r\rfloor }_q\!(\Omega)}.
	\end{equation}
	
	It remains to consider the terms in \eqref{fullsum} with $d({\bf i},{\bf j})=1$. In this case $\omega_{\bf i}$ and $\omega_{\bf j}$ have a common face of
	dimension between $0$ and $d-1$. 
	We alter
	variables under the integral over $\omega_{\bf j}$ by substituting for ${\bf y}$ its reflection ${\bf u}$ with respect to $\omega_{\bf i} \cap \omega_{\bf j}$. Evidently,
	$\varphi_{\bf j}({\bf y})=\varphi_{\bf i}({\bf u})$ and $|{\bf x} - {\bf y}|\geqslant |{\bf x} - {\bf u}|$, ${\bf x}\in\omega_{\bf i}$. 
	Hence, for all ${\bf i},{\bf j}$ with $d({\bf i},{\bf j})=1$,
	\begin{align*}
		\int_{\omega_{\bf i}} \int_{\omega_{\bf j}} \frac{\left|D^{{\bf k}} \varphi_{\bf i}({\bf x}) - D^{{\bf k}} \varphi_{\bf j}({\bf y})\right|^q}{\left| {\bf x}-{\bf y}\right|^{\left\{r\right\}q+d}}\,{\rm d}{\bf x}{\rm d}{\bf y}
		&\leqslant
		\int_{\omega_{\bf i}} \int_{\omega_{\bf i}} \frac{\left|D^{{\bf k}} \varphi_{\bf i}({\bf x}) - D^{{\bf k}} \varphi_{\bf i}({\bf u})\right|^q}{\left| {\bf x}-{\bf u}\right|^{\left\{r\right\}q+d}}\,{\rm d}{\bf x}{\rm d}{\bf u}
		\\ &\leqslant m^{rq-d} \left|\varphi\right|_{W_q^r(\Omega)}^q. 
	\end{align*}
	Since
	$$
	\sum\limits_{\ell=0}^{d-1}\!\!\!\! \sum\limits_{{\bf i},{\bf j}\in \mathcal{N}_m^d\atop d({\bf i},{\bf j})=1,\, \dim(\omega_{\bf i} \cap \omega_{\bf j})=\ell} \!\!\!\!\!\!\!m^{rq-d} 
	\leqslant \sum\limits_{{\bf i}\in \mathcal{N}_m^d} \left(3^d - 1\right) m^{rq-d} 
	= \left(3^d - 1\right) m^{rq},
	$$
	we arrive at
	\begin{equation}
		\label{estimatesum2}
		\sum\limits_{{\bf i},{\bf j}\in \mathcal{N}_m^d\atop d({\bf i},{\bf j})=1} 
		\sup\limits_{{\bf k}\in \mathbb{Z}_+^d\atop |{\bf k}|=\lfloor r\rfloor } 
		\int_{\omega_{\bf i}} \int_{\omega_{\bf j}} \frac{\left|D^{{\bf k}} \varphi_{\bf i}({\bf x}) - D^{{\bf k}} \varphi_{\bf j}({\bf y})\right|^q}{\left| {\bf x}-{\bf y} \right|^{\left\{r\right\}q+d}}\,{\rm d}{\bf x}{\rm d}{\bf y}
		\leqslant\left(3^d - 1\right) m^{rq} \left|\varphi\right|_{W_q^r(\Omega)}^q.
	\end{equation}	
	
	Combining~\eqref{estimatesum1},~\eqref{estimatesum3} and \eqref{estimatesum2} we obtain 
	\begin{equation}\label{ue}
		|f_m|_{W^r_q(\Omega)} \leqslant  C_6 m^{r},
	\end{equation}
	where $C_6 := \left(3^d\left|\varphi\right|_{W_q^r(\Omega)}^q +\frac{2^{q+\{r\}+d}\mu_d}{\{r\}q} |\varphi|^q_{W^{\lfloor r\rfloor }_q(\Omega)}\right)^{\frac 1q}$ 
	is independent of $m$. Taking into account \eqref{norm_fmk} we obtain \eqref{norm_fm}. 
		
	Finally, we consider the function $\displaystyle g_m := C_6^{-1}m^{-r} f_m$ and set $N_m := m^d$ and $\varepsilon_m :=\left\|g_m\right\|_{L_\infty(\Omega)}
	=\displaystyle C_6^{-1} m^{-r} \|\varphi\|_{L_\infty(\Omega)}$. 
	Let us show that
	$\displaystyle{\inf\limits_{\triangle \in \mathfrak{D}_{N_m}} \inf\limits_{s\in \mathcal{S}_0\left(\triangle\right)} \|g_m - s\|_{L_\infty(\Omega)} >\frac{\varepsilon_m}{3}}$. Then \eqref{lbeq} will follow in view of \eqref{norm_fm}.
	
	Let $\triangle \in \mathfrak{D}_{N_m}$ be any convex partition comprising at most $N_m$ cells. 
	Assume to the contrary to our claim that $ \|g_m - s\|_{L_\infty(\Omega)}
	\leqslant \frac{\varepsilon_m}{3}$ for some $s\in  \mathcal{S}_0\left(\triangle\right)$.  
	Consider a cell $\delta\in\Delta$ whose closure $\bar\delta$ contains the center ${\bf c}$ of a cube $\omega\in\square_m$. Then $\bar\delta$ is located completely inside $\omega$
	as otherwise we can find a point ${\bf c}'\in\delta$ sufficiently close to the boundary of $\omega$ and a point ${\bf c}''\in\delta$ sufficiently close to  ${\bf c}$, 
	such that $s({\bf c}')=s({\bf c}'')$ and $g_m\left({\bf c}'\right) <\min\limits_{{\bf x}\in\Omega} g_m({\bf x}) + \frac{\varepsilon_m}{12}$, and $g_m\left({\bf c}''\right) > \max\limits_{{\bf x}\in \Omega} g_m({\bf x}) - \frac{\varepsilon_m}{12}$, and hence
	\begin{align*}
		\|g_m - s\|_{L_\infty(\Omega)}
		& \geqslant \max \big\{|g_m({\bf c}') - s({\bf c}') |;|g_m({\bf c}'') - s({\bf c}'') | \big\}\\
		& > \frac12\Big({\max\limits_{{\bf x} \in \Omega} g_m({\bf x})- \min\limits_{{\bf x} \in \Omega} g_m({\bf x})}\Big)- \frac{\varepsilon_m}{6} 
		=\frac{\varepsilon_m}{3},
	\end{align*}
	which contradicts the assumption. Therefore in each of $m^d$ cubes $\omega\in\square_m$ there is at least one cell $\delta$ of $\triangle$ such that
	$\bar\delta$ is located completely inside that cube.
	It follows that the total number of cells in $\triangle$ is greater than $m^d = N_m$, a contradiction to the choice of the partition. $\qed$
\end{proof}

\section{Acknowledgments}

The work of Oleksandr Kozynenko was supported in part by a DAAD Scholarship in the programme Research Grants 
-- Bi-nationally Supervised Doctoral Degrees, 2016/17 (57214225).

\end{document}